\documentclass[a4paper,11pt,reqno]{amsart}

\usepackage{mathrsfs}
\usepackage{nicefrac}
\usepackage{a4wide}
\usepackage[applemac]{inputenc}
 \usepackage[pdfpagemode=UseNone,bookmarksopen=false,colorlinks=true]{hyperref}
\usepackage{enumerate}
\usepackage{bbm}

\hypersetup{ colorlinks=true, linkcolor=blue, citecolor=blue,
  filecolor=blue, urlcolor=blue}

\usepackage[comma, sort&compress]{natbib}
\usepackage{graphicx}
\usepackage{mathrsfs}
\usepackage{amsmath}
\usepackage{amsthm}
\usepackage{amssymb}
\usepackage{color}
\definecolor{mygray}{gray}{0.6}
\usepackage[normalem]{ulem} 
\usepackage {ulem} 
\usepackage[osf,sc]{mathpazo}

\usepackage{enumerate}

\usepackage{url}

\theoremstyle{plain}
\newtheorem{theorem}{Theorem}[section]
\newtheorem{proposition}[theorem]{Proposition}
\newtheorem{lemma}[theorem]{Lemma}

\newtheorem{definition}[theorem]{Definition}

\theoremstyle{definition}

\theoremstyle{remark}
\newtheorem{remark}[theorem]{Remark}

\newcommand{\N}{\ensuremath{\mathbb{N}}}
\newcommand{\Z}{\ensuremath{\mathbb{Z}}}
\newcommand{\R}{\ensuremath{\mathbb{R}}}

\newcommand{\cov}{\ensuremath{{\rm{cov}}}}
\newcommand{\be}{\begin{equation}}
\newcommand{\ee}{\end{equation}}
\newcommand{\e}{\mathrm e}
\newcommand{\De}{\mathrm d}
\renewcommand{\c}{\mathrm c}
\newcommand{\vr}{\varphi}
\newcommand{\eps}{\varepsilon}
\newcommand{\one}[1]{\mathbf{1}_{\left\{ #1 \right\}}}

\newcommand{\ind}{\mathbf{1}}

\newcommand{\eq}[1]{\begin{equation#1}}
\newcommand{\eeq}[1]{\end{equation#1}}
\newcommand{\eqa}[1]{\begin{eqnarray#1}}
\newcommand{\eeqa}[1]{\end{eqnarray#1}}
\numberwithin{equation}{section}

\title[Fast decay of covariances under $\delta-$pinning in the membrane model]{Fast decay of covariances under $\delta-$pinning in the critical and supercritical membrane model}

\author[E. Bolthausen]{Erwin Bolthausen}
\address{Institut f\"ur Mathematik, Universit\"at Z\"urich, Winterthurerstrasse 190, CH-8057, Zurich, Switzerland}
\email{eb@math.uzh.ch}

\author[A. Cipriani]{Alessandra Cipriani}
\address{Weierstrass Institute, Mohrenstrasse 39, 10117 Berlin, Germany}
\email{Alessandra.Cipriani@wias-berlin.de}

\author[N. Kurt]{Noemi Kurt}
\address{Technische Universit\"at Berlin,
MA 766, Strasse des 17. Juni 136, 10623
Berlin, Germany}
\email{kurt@math.tu-berlin.de}

\date{\today}

\begin{document}

\begin{abstract}
We consider the membrane model, that is the centered Gaussian field on $\Z^d$ whose covariance matrix is given by the inverse of the discrete Bilaplacian. We impose a $\delta-$pinning condition, giving a reward of strength $\varepsilon$ for the field to be $0$ at any site of the lattice. In this paper we prove that in dimensions $d\geq 4$ covariances of the pinned field decay at least stretched-exponentially, as opposed to the field without pinning, where the decay is polynomial in $d\geq 5$ and logarithmic in $d=4.$ The proof is based on estimates for certain discrete Sobolev norms, and on a Bernoulli domination result.
\end{abstract}
\maketitle

\section{The model and main results} 
The membrane model, or Laplacian model, is an example of an effective random interface, see for example \cite{Sakagawa}, \cite{velenikloc} and \cite{Kurt_thesis}. We will work on the $d$-dimensional integer lattice $\Z^d$, and in the present paper our focus will be in $d\geq 4$, although the definition is well-posed in all dimensions. For $N\in\N,$ let $V_N:=[-N/2,\,N/2]^d\cap \Z^d$ and $V_N^\c:=\Z^d\setminus V_N$. The discrete Laplacian $\Delta$ on $\Z^d$ is defined as the operator acting on functions $f:\Z^d\to\R$ by
\[\Delta f(x)=\frac{1}{2d}\sum_{y:\,\|x-y\|=1}\big(f(y)-f(x)\big),\]
where $\|x\|$ denotes the $\ell^1$-norm on the lattice. We sometimes write $f_x$ for $f(x).$

\begin{definition} The \emph{membrane model} is the random field $\{\varphi_x\}_{x\in\Z^d}\in\R^{\Z^d}$ with zero boundary conditions outside $V_N$, whose distribution is given by
\begin{equation}
 \label{def:field}
 P_N(\De\varphi)=\frac{1}{Z_N}\exp\left(-\frac{1}{2}\sum_{x\in \Z^d}(\Delta\varphi_x)^2\right)\prod_{x\in V_N}\De\varphi_x
\prod_{x\in V_N^{\mathrm c}}\delta_0(\De\varphi_x),
\end{equation}
where $Z_N$ is a normalizing constant. 
\end{definition}
Note that by re-summation, the law $ P_N$ of the field is the law of the centered Gaussian field on $V_N$ with covariance matrix
$$G_N(x,y):=\cov(\varphi_x,\varphi_y)=\left(\Delta_N^2\right)^{-1}(x,y), \quad x,y\in V_N.$$
Here, $\Delta_N^2=\big(\Delta^2(x,y)\big)_{\{x,\,y\in V_N\}}$ is the Bilaplacian with $0$-boundary conditions outside $V_N$. We extend both $\Delta^2_N$ and $G_N$ to $x,y\in \Z^d$ by setting the entries to $0$ outside $V_N\times V_N.$ For $x\in V_N,$ the matrix $G_N$ is determined by the boundary value problem \footnote{$\delta_{x}(y)$ is the Dirac delta mass at $x$, i. e., $\delta_{x}(y)=1\iff x=y$.}
\begin{equation}\left\{
\begin{array}{lr}
\Delta^2 G_N(x,y)=\delta_{x}(y),& y \in V_N\label{eq:uno} \\
G_N(x,y)=0, & y \in \partial_2 V_N,\nonumber
\end{array}\right.
\end{equation}
where we denote $\partial_2V_N:=\{y\in V_N^{\mathrm c}:\,\exists z\in V_N:\,  \|y-z\|\leq 2\}$. It is known that in $d\ge 5$ there exists $P$ on $\R^{\Z^d}$ such that $P_N\to P$ weakly (\cite{Sakagawa}).  Under $P$, the canonical coordinates $(\vr_x)_{x\in \Z^d}$ form a centered Gaussian process with covariance given by
\begin{equation}G(x,y)= \Delta^{-2}(x,y)= \sum_{z\in \Z^d} \Delta^{-1}(x,z)\Delta^{-1}(z,y)= \sum_{z\in \Z^d} \Gamma(x,z) \Gamma(z,y),\label{eq:GtoGamma}\end{equation}
where $\Gamma$ denotes the covariance of the discrete Gaussian Free Field (DGFF, see \citet[Section 2]{ASS} for an overview). The matrix $\Gamma$ has an easy representation in terms of the simple random walk $(S_n)_{n\ge 0}$ on $\Z^d$ given by
$$\Gamma(x, \,y)=\sum_{m\ge 0} \mathrm P_x[ S_m=y]$$
($\mathrm P_x$ is the law of $S$ starting at $x$). This entails that
\begin{equation}
G(x,\,y)= \sum_{m\ge 0} (m+1) \mathrm P_x[ S_m=y]=\mathrm E_{x,y }\left[ \sum_{\ell,\,m=0}^{+\infty} \ind_{\left\{S_m= \tilde S_\ell\right\} }\right]\label{eq:RW-repG}
\end{equation}
where $S$ and $\tilde S$ are two independent simple random walks started at $x$ and $y$ respectively.  One can note from this representation that $G(\cdot,\,\cdot)$ is translation invariant. The existence of the infinite volume measure in $d\ge 5$ gives that $G(0,0)<+\infty$. Using the above one can derive the following property of the covariance:
\begin{lemma}[{\citet[Lemma 5.1]{Sakagawa}}]\label{lemma: covariance:mm}
Let $d\ge 5$. Then
\eq{}\label{eq:cov:mm}
\lim_{\|x\|\to+\infty}\frac{G(0,x)}{\|x\|^{4-d}}=\eta
\eeq{}
where
$$
\eta=(2\pi)^{-d}\int_0^{+\infty}\int_{\R^d}\exp\left(\iota\langle \zeta,\, \theta\rangle-\frac{\|\theta\|^4 t}{4 d^2}\right)\De\theta\De t
$$
for any $\zeta\in \mathbb{S}^{d-1}$ and $\iota=\sqrt{-1}$.
\end{lemma}
In other words, as $\|x-y\|\to\infty,$ the covariance between $\varphi_x$ and $\varphi_y$ decays like $\|x-y\|^{4-d}$ in the supercritical dimensions. For $d=4$ it was shown that $G_N(x,y)$ behaves in first order as $\gamma_4(\log N-\log\|x-y\|)$ for some $\gamma_4\in (0,\infty),$ if $x$ and $y$ are not too close to the boundary of $V_N,$ see \citet[Lemma 2.1]{Cip13}.

\medskip

The goal of this paper will be to show that this polynomial decay of covariances changes drastically if we introduce a so-called ``$\delta$-pinning'' which gives a reward of size $\varepsilon>0$ if the interface touches the $0$-hyperplane at a site $x\in\Z^d$. More precisely, we introduce an atom of size $\varepsilon$ in $0$ to our model \eqref{def:field}:

%

\begin{definition} Let $\varepsilon>0$ and let $P_N$ be defined as in \eqref{def:field}. The \emph{membrane model on $V_N$ with pinning of strenght $\eps$} is defined as 
\begin{equation}P_N^\varepsilon(\mathrm d \varphi)=\frac{1}{Z_N^\eps} \exp \left(-\frac{1}{2}\sum_{x \in \Z^d}\varphi_x \Delta^2 \varphi_x \right) \prod_{x \in V_N}\left(\mathrm d \varphi_x+\eps\delta_0(\mathrm d \varphi_x)\right) \prod_{x \in V_N^{\c}}\delta_0(\mathrm d \varphi_x).\label{def:pinned}\end{equation}
\end{definition}
With this definition we have for any measurable function $f:\R^{\Z^d}\to \R$,
\begin{eqnarray*}E_N^\varepsilon (f)&=&\frac{1}{Z_N^\eps}\int f(\vr)\exp \left(-\frac{1}{2}\sum_{x \in \Z^d}\varphi_x \Delta^2 \varphi_x \right) \prod_{x \in V_N}\left(\mathrm d \varphi_x+\eps\delta_0(\mathrm d \varphi_x)\right) \prod_{x \in V_N^{\c}}\delta_0(\mathrm d \varphi_x)= \\
&&=\sum_{A \subseteq V_N}\eps^{|A|}\frac{Z_{V_N\setminus A}}{Z_N^\eps}E_{V_N\setminus A}(f)\end{eqnarray*}
where $E_{V_N\setminus A}$ is the mean according to the measure $P_{V_N\setminus A}$ defined for $A\subseteq V_N$ by
$$
P_{V_N\setminus A}(\De \vr)=\frac{1}{Z_{V_N\setminus A}}\int \exp \left(-\frac{1}{2}\sum_{x \in \Z^d}\varphi_x \Delta^2 \varphi_x \right) \prod_{x \in V_N\setminus A}\mathrm d \varphi_x\prod_{x \in A\cup V_N^c}\delta_0(\mathrm d \varphi_x).
$$
Thus $P_N^\varepsilon$ is a convex combination of probabilities $P_{V_N\setminus A}$ 
which are distributed according to a probability measure on $\mathcal P(V_N)$\footnote{$\mathcal P(A)$ is the powerset of $A\subset \Z^d$.}, namely
$$\zeta^\eps_N(A)=\zeta^\eps_N(\mathcal A =A):=\eps^{|A|}\frac{Z_{V_N\setminus A}}{Z_N^\eps}$$
\cite[Section 5]{velenikloc}.
Here and in the following $\mathcal A$ denotes a $\mathcal P(V_N)$-valued random variable under some site percolation law (which will be specified in each occurrence). 
Using the above expansion, we obtain for the covariances with respect to $P_N^\eps$
\be \label{eq:mixture1}E_N^\eps[\vr_x\vr_y]=\sum_{A \subseteq V_N}\zeta_N^\eps(A)E_{V_N\setminus A}[\varphi_x\varphi_y].\ee

To write this even more concisely, let $A\subset \Z^d$ with $|A^\c|<+\infty$, and denote by $P_{A^c}$ the law of the membrane model with $0-$boundary conditions outside $A^\c.$ Let 
\[G_A(x,y):=E_{A^\c}[\varphi_x,\varphi_y], \quad x,\,y\in A^\c,\]
which we again extend by setting it to $0$ to all of $\Z^d.$ Observe that in this notation $G_N=G_{V_N^\c}.$
Then we can rewrite \eqref{eq:mixture1} as
\be \label{eq:mixture}E_N^\eps[\vr_x\vr_y]=E_{\zeta_N^\eps}\left[G_{\mathcal A\cup V_N^\c}[\varphi_x\varphi_y]\right].\ee

Our main result shows, in the following couple of theorems, that for any positive pinning strength $\eps$ the correlations between two points decay at least stretched-exponentially in the distance.

\begin{theorem}[Decay of covariances, supercritical case]\label{thm:random}
Let $d\geq 5$ and $\varepsilon >0.$ There exists $\alpha>0$ independent of $\varepsilon$ such that
\begin{equation}
\limsup_{\|x-y\|\to+\infty}\limsup_{N\to+\infty}E^\eps_N[\varphi_x\varphi_y]\e^{\|x-y\|^\alpha}=0.\label{eq:decay}
\end{equation}
\end{theorem}

\begin{theorem}[Decay of covariances, critical case]\label{thm:random_d4}
Let $d= 4$ and $\varepsilon >0.$ For every $0<\lambda\leq 1$ there exists $\beta=\beta(\lambda)>0$ independent of $\varepsilon$ such that for $\delta\in (0,1]$
\begin{equation}
\limsup_{N\to+\infty}\sup_{x,y\in V_N: \|x-y\|\geq \delta N^\lambda}E^\eps_N[\varphi_x\varphi_y]\e^{\|x-y\|^\beta}=0.\label{eq:decay_d4}
\end{equation}
\end{theorem}
This result complements the one of \cite{Sak2011}, who proves, via a free-energy estimate, that in $d\ge 4$ the model is localized, in the sense that it exhibits a positive density of pinned sites.

The proof relies on two main steps: firstly, using certain equivalences of discrete Sobolev norms, we show in Theorem~\ref{thm:deterministic} that for ``very good sets'' $A$ the decay is indeed exponential:
\[
\left| G_A(x,\,y)\right|\le c \e^{-c'\|x-y\|}.
\]
Unfortunately these sets do not have probability high enough under $\zeta^\epsilon_N$, thus we need to make adjustments to the definition of ``very good'' to balance the effect of the random environment of pinned points and the exponential decay. 

For the DGFF it was proved (see \cite{BoltBrydges, BoltVel,IofVel,DeuschelVelenik}) that the decay of the covariances is in fact exponential in the critical and supercritical dimensions. We conjecture that this is also true for the membrane model, but due to the lack of the random walk representation (see Remark~\ref{rem:important} below) we are not able to prove this at the moment. Results on the membrane model with pinning were shown in $(1+1)$ dimensions by \cite{CaravennaDeuschel_pin}.

The structure of the paper is as follows: we begin with general results, including Bernoulli domination, in Section~\ref{sec:general}. In Section~\ref{sec:main} we prove our main theorems, starting with Theorems~\ref{thm:deterministic}, \ref{thm:deterministic_d4} in Subsection~\ref{subsec:det_det}, and then Theorems~\ref{thm:random}, \ref{thm:random_d4} in Subsection~\ref{subsec:det_to_ran}.

\section{General results on the membrane model}\label{sec:general}
In this section we collect and prove some results on the membrane model that will be important for the proof of the main results. 
Just as the DGFF enjoys the spatial Markov property, the membrane model does too. In fact it holds that
\begin{proposition}[Markov property, {\citet[Lemma 2.2]{Cip13}}]\label{prop:MP}
Let $(\vr_x)_{x\in\Z^d}$ be the membrane model under the measure $P_N$. Let $B \subseteq V_N$. Let $\mathcal F_B:=\sigma(\vr_z,\,z \in V_N\setminus B)$. Then
\eq{}\label{eq:cip13}
 \{\vr_x\}_{x\in B} \stackrel{d}{=} \left\{\mathbf E_N\left[\vr_x|\mathcal F_B\right]+ \vr'_x\right\}_{x\in B}
\eeq{}
where ``$\stackrel{d}{=}$'' indicates equality in distribution. In particular, under $\mathbf P_N(\cdot)$, $\vr'_x$ is independent of $\mathcal F_B$. Also $\{\vr'_x\}_{x \in B}$ is distributed as the membrane model with $0$-boundary conditions outside B.
\end{proposition}

A further important observation is that the variances of the membrane model are decreasing in the number of points in which the field is $0.$ 
\begin{lemma} \label{lemma:monvar}
Let $A_1\subset A_2\subset V_N.$ Then
$$G_{A_2\cup V_N^\c}(x,x)\leq G_{A_1\cup V_N^\c}(x,x)\leq G_N(x,x)$$ for all $x\in V_N\setminus A_2.$
\end{lemma}
\begin{proof} Let $B:=V_N\setminus A_1.$ By Proposition~\ref{prop:MP}, for a membrane model $\vr$ under $P_N$
\[
 \{\vr_x\}_{x\in B} \stackrel{d}{=} \left\{E_N\left[\vr_x|\mathcal F_B\right]+ \vr'_x\right\}_{x\in V_N\setminus B}
\]
where $\vr'$ has the law of a membrane model on $B$ with zero boundary conditions on $A_1\cup V_N^\c$. Therefore
\[
G_N(x,\,x)-G_{A_1\cup V_N^c}(x,\,x)=E_N\left[(E_N\left[\vr_x|\mathcal F_B\right])^2\right]\ge 0.
\]
For $A_1\subset A_2$, the proof follows exactly the same lines replacing $V_N$ with $V_N\setminus A_1$ above.
\end{proof}

Next we prove that $G_A$ satisfies a similar boundary value problem as $G_N.$
\begin{lemma}\label{lem:G_A}
Let $d\geq 4, A\subset\Z^d$ such that $|A^\c|<+\infty.$ Let $N$ be large enough such that $A^{\c}\subset V_N,$ and fix $x\in A^{\c}.$ Then $G_A(x,y)$ solves the discrete boundary value problem 
\eq{}\left\{\begin{array}{lr}\label{eq:bv-prob}
\Delta^2 G_A(x,y)= \delta_x(y) & y\in A^{\c},\\
G_A(x,y)=0 & y\in A\cup V_N^c.
\end{array}\right.\eeq{}
Moreover, there exists a constant $\gamma=\gamma(d)$ such that for all $x\in \Z^d,$
\begin{equation}
\label{eq:boundC}G_A(x,x)\leq \begin{cases}\gamma &\mbox{ if }d\geq 5,\\
 \gamma\log N &\mbox{ if }d=4.\end{cases}
\end{equation}
\end{lemma}

\begin{proof} 
By Proposition \ref{prop:MP}, $G_A$ is the covariance matrix of the membrane model on $V_N$ conditioned to be $0$ in $A\cup V_N^\c.$ A well-known fact about Gaussian random vectors is that conditioning on the values of some of the entries yields again a Gaussian vector, whose covariance matrix can be calculated by a simple formula. In our case, this formula looks as follows: let 
\[\Sigma_{A,N}:=(G_N(x,y))_{x,y\in A\cup V_N^\c}.\]
Then  \cite[Chapter 6]{zhang}
\be \label{eq:schur} 
G_A(x,y)=G_N(x,y)-\sum_{z,w\in A\cup V_N^c}G_N(x,z)\Sigma_{A,N}^{-1}(z,w)G_N(w,y).\ee
From \eqref{eq:uno} we immediately obtain \eqref{eq:bv-prob}, and using the fact that $G_A$ is positive semi-definite (since it is a covariance matrix) and \citet[Proposition 2.1.1 resp. Proposition 2.1.2]{Kurt_thesis} we get \eqref{eq:boundC}.
\end{proof}
For $d\geq 5$ we obtain the same result for any $A\subseteq \Z^d.$ 
\begin{lemma}\label{lem:G_A_d4}
Let $d\geq 5, A\subset \Z^d,$ and $x\in A^{\c}$ (thus $A^\c$ is possibly infinite). The membrane model on $A^c$ is well-defined, and its covariance matrix $G_A(x,y)$ solves the discrete boundary value problem 
\eq{}\left\{\begin{array}{lr}\label{eq:bv-prob_d4}
\Delta^2 G_A(x,y)=  \delta_x(y)& y\in A^{\c},\\
G_A(x,y)=0 & y\in A.
\end{array}\right.\eeq{}
Moreover, there exists a constant $\gamma=\gamma(d)$ such that 
\[G_A(x,x)\leq \gamma\]
for all $x\in \Z^d.$
\end{lemma}
\begin{proof}
By Lemma \ref{lemma:monvar}, $G_A(x,x):=\lim_{N\to+\infty}G_{A\cup V_N^\c}(x,x)$ exists for $x\in\Z^d,$ and from \eqref{eq:boundC} we know that the sequence of measures $P_{A^\c\cap V_N}$ is tight. Since we are dealing with Gaussian measures, it is enough to prove the existence of the weak limit $P_{A^\c}$ of $P_{A^\c\cap V_N}$ to show the statement. Then \eqref{eq:bv-prob_d4} follows by taking limits in \eqref{eq:bv-prob}.
\end{proof}

\begin{remark}\label{rem:important}
At this point it is important to note that $G_A$ is \emph{not} the convolution of the covariance matrix of the DGFF with $0$-boundary conditions outside $A^\c,$ which is only the case for the infinite volume situation, c.~f. \eqref{eq:GtoGamma}. Therefore the random walk representation \eqref{eq:RW-repG} doesn't carry over to $G_A.$ This is an important difference between the membrane model and the DGFF. To study properties of the pinned DGFF one can rely on the random walk representation, as for example \cite{velenikloc, BoltVel, IofVel, BoltBrydges, CoqMil} do. In the membrane model one can, as in \cite{Cip13} and \cite{Kurt_d4}, approximate $G_{N}$ by a random walk representation and thus derive useful estimates. However, this approximation is only valid for convex connected $A^\c,$ and thus cannot be applied to the pinning case. We therefore need to apply very different methods in order to find estimates for $G_A(x,y)$ for general $A\subset V_N.$ Our approach is based on equivalences of certain discrete Sobolev norms and a Bernoulli domination argument, with which we begin.
\end{remark}
\subsection{The random environment of pinned points}
Let us now prove a simple Lemma on partition functions for the measure $P_N^\eps$. We denote as $f_{\vr_E }$ the density of $\vr_x$ with respect to the measure $\prod_{x\in E}\De \vr_x\prod_{x\in E^\c}\delta_0\left(\De \vr_x\right)$ and $Z_{E}$ its partition function.
\begin{lemma}\label{lem:partition} In $d\geq 5$ there exist constants $0<{C}_\ell,\,{C}_r<+\infty$ such that for every $E\subseteq V_N$ and $x\in E$ 
\begin{equation}\label{eq:boundZ}{C}_\ell\le \frac{Z_E}{Z_{E\setminus\{x\}}}\leq {C}_r.\end{equation}
\end{lemma}
\proof 
\begin{equation*}\label{eq:ratio_two_Z}\frac{Z_E}{Z_{E\setminus\{x\}}}=\frac{f_{\vr_E}(0,\ldots,0)}{f_{\vr_{E\setminus\{x\}}}(0,\ldots,0)}=
f_{\vr_{x}|\vr_{E\setminus\{x\}}}(0|0,\ldots,0) \end{equation*}
where the latter is the conditional density of $\vr_x$ given that the field $\left\{\vr_{x},\,x\in E\setminus \{x\}\right\}$ is zero. We know already that $\vr_{x}$ conditioned on $\left\{\vr_{x},\,x\in E\setminus \{x\}\right\}$ is a well-defined normal variable $ \mathcal N(0,\,\sigma_{x}^2)$ by Proposition \ref{prop:MP}, with $\sigma^2_x\leq \gamma$ because of Lemma \ref{lem:G_A}. Therefore
\[
0< {C}_\ell:=\frac{1}{\sqrt{2\pi\gamma}}\le \frac{1}{\sqrt{2\pi \sigma_x^2}}=  f_{\vr_{x}|\vr_{E\setminus\{x\}}}(0|0,\ldots,0)\le 1=:{C}_r.
\]
\endproof
\begin{lemma}\label{lem:partition_d4} In $d\geq 5$ there exist constants $0<{C}_\ell,\,{C}_r<+\infty$ such that for every $E\subseteq V_N$ and $x\in E$ 
\begin{equation}\label{eq:boundZ_d4}{C}_\ell\frac{1}{\sqrt{\log N}}\le \frac{Z_E}{Z_{E\setminus\{x\}}}\leq {C}_r.\end{equation}
\end{lemma}
\proof
The proof is similar to Lemma~\ref{lem:partition}, using the fact that $\sigma_x^2\leq \gamma \log N$, see Lemma \ref{lem:G_A_d4}.
\endproof

Our target now is to control the pinning measure $\zeta_N^\eps$ through a natural distribution of sites on the discrete lattice, that is through independent site percolation. We will briefly recall here two definitions.
\begin{definition}[Stochastic and strong stochastic domination]
Given two probability measures $\mu$ and $\nu$ on the set $\mathcal P(\Lambda)$, $|\Lambda|<+\infty$, we will say that $\mu$ \emph{dominates} $\nu$ \emph{strongly stochastically} 
if for all $x$, $C \subseteq \Lambda\setminus \{x\}$,
\begin{equation}\label{eq:ssd}
 \mu(A:\,x\in A\,|\,A\setminus\{x\}=C)\geq \nu(A:\,x\in A\,|\,A\setminus\{x\}=C).
\end{equation}
When \eqref{eq:ssd} holds we will write $\mu \succ \nu$. We will say that $\mu$ \emph{dominates} $\nu$ \emph{stochastically}, $\mu \succeq \nu$,
if for all increasing functions $f$,
\begin{equation*}
 \mu(f)\geq \nu(f).
\end{equation*}
Note that strong stochastic domination implies stochastic domination.
\end{definition}
Let now $\nu_\Lambda^\rho$ be the Bernoulli site percolation measure on $V_N$ with intensity $\rho$. We would like to prove that our Gaussian free fields restricted to the pinned 
set are ``sandwiched'' between two such Bernoullian in the stochastic ordering. This argument is similar to the one in \citet[Section 5.3]{velenikloc}.
 \begin{proposition}\label{prop:dom}
Let $d\geq 5.$ There exist constants $0<c_-(d)<c_+(d)<\infty$ such that for $\eps$ small enough,
$$\nu_{N}^{\rho_-(d)}\prec \zeta^\eps_{N} \prec \nu_{N}^{\rho_+(d)}$$
where $\rho_\pm(d)=c_\pm(d) \eps$.
\end{proposition}
\proof In the following we will omit the subscript $N$ as we will be always working on the $d$-dimensional box of side-length $N$. The first step is to notice that for all $i\in V_N$, 
$C\subseteq V_N\setminus\{i\}$,
$$\zeta^\eps(A:\,i\in A\,|\,A\setminus\{i\}=C)=\frac{\zeta^\eps(C\cup\{i\})}{\zeta^\eps(C)}$$
and by \eqref{eq:boundZ}
\[
{C}_\ell{\leq} \frac{Z_{C\cup\{i\}}}{Z_C}{\leq} {C}_r.
\]
Therefore stochastic domination is achieved for two Bernoulli measures of parameter $\rho_-(d):={C}_\ell \eps$, $\rho_+(d):={C}_r \eps$. 
\endproof

\begin{proposition}\label{prop:dom_d4}
Let $d=4.$ There exist constants $0<c_-(4)<c_+(4)<\infty$ such that for $\eps$ small enough,
$$\nu_{N}^{\rho_-(4)}\prec \zeta^\eps_{N} \prec \nu_{N}^{\rho_+(4)}$$
where $\rho_+(4)=c_+(4) \eps$, and 
\[\rho_-(4)=\frac{c_-(4)\eps}{\sqrt{\log N}}.\]
\end{proposition}
\begin{remark}
Observe that $\rho_{-}(4)$ converges to $0$ as $N\to+\infty.$
\end{remark}
\begin{proof}
The argument is the same of Prop.~\ref{prop:dom} where the conclusion is this time drawn from \eqref{eq:boundZ_d4}.
\end{proof}

\section{Proof of the main results}\label{sec:main}
\subsection{Equivalence of norms}
For a function $f:\Z^d\to \R$ we define the derivative in the $i$-th coordinate direction, $i\in\{1,...,d\}$ by
\[ D_if(x):=f(x+e_i)-f(x), \;x\in \Z^d,\; i=1,\,\ldots,\,d,\]
where $e_i$ is the unit vector in direction $i.$ Define the discrete gradient as 
\[\nabla f(x):=(D_1f(x),\,\ldots,\,D_df(x)).\]
It will be convenient to introduce $D_{-i}f(x):=f(x-e_i)-f(x)=-D_if(x-e_i),$ for $i=1,\,\ldots,\,d.$ The second discrete derivatives of a function are
\[D_{ij}f(x):=D_iD_jf(x), \quad i,\,j\in\{\pm 1,\,\ldots,\,\pm d\}.\]
With this notation, the discrete Laplacian is then given by
\begin{equation*}\label{eq:Lapl}\Delta f(x)= -\frac{1}{2d}\sum_{i=1}^d D_{i,\,-i}f(x)\end{equation*}
and the Bilaplacian assumes the form
\be\label{eq:Bilapl} \Delta^2 f(x)=\frac{1}{4d^2}\sum_{i,\,j=1}^d D_{i,\,-i}D_{j,\,-j}f(x).\ee

We have the following summation by parts formula whose proof is an elementary calculation:

\begin{lemma}\label{lem:sumbyparts}
Let $f,g$ be such that $\sum_{x\in\Z^d}f(x)g(x)<+\infty$ and $\sum_{x\in\Z^d}f(x)g(x+e_i)<+\infty$ for all $i\in\{\pm 1...\pm d\}.$ Then for all $i\in\{\pm 1,...,\pm d\}$ we have
\[ \sum_{x\in\Z^d}D_if(x)g(x)=\sum_{x\in\Z^d}f(x)D_{-i}g(x).\]
\end{lemma}
Moreover
\begin{lemma}\label{lem:laplace-norm}
For $u:\Z^d\to\R$ we have 
\[\sum_{x\in\Z^d}\sum_{i,j=1}^d (D_iD_ju(x))^2=4d^2\sum_{x\in\Z^d}u(x)\, \Delta^2u(x) .\]
\end{lemma}
\begin{proof}
Follows from Lemma \ref{lem:sumbyparts} and \eqref{eq:Bilapl}.
\end{proof}

The standard discrete Sobolev norms on $E\subseteq\Z^d$ associated to the discrete Sobolev space $H^k(E)$ are given by
\be \|f\|_{H^k(E)}^2=\sum_{\ell=0}^k\left(\sum_{i_1,...,i_\ell=1}^d\sum_{x\in E}|D_{i_1}...D_{i_\ell}f(x)|^2\right).\ee

\medskip

We also introduce the norms

\be \|\nabla_k f\|_{L^2(E)}^2:=\sum_{i_1,...,i_k}\sum_{x\in E}|D_{i_1}...D_{i_k}f(x)|^2.\ee

We obviously have 
\eq{}\label{eq:norm_1}\|\nabla_k f\|_{L^2(E)}\leq \|f\|_{H^\ell(E)},\quad k\leq \ell\eeq{}
and
\eq{}\label{eq:norm_2}
\|\Delta f(x)\|_{L^2(E)}\leq C\|\nabla_2f\|_{L^2(E)}\eeq{}
for some $C$ depending only on $d.$
The next Lemma will show that the above norms are equivalent on subsets where ``groups'' of pinned points are not too spread out. Let $A\subset \Z^d$. Set
\[
\widehat A:=\left\{x\in A:\,\text{for all }y\sim x,\,y\in A\right\}.
\]
We can think of $\widehat{A},$ which obviously is a subset of $A,$ as the interiour of ``pinned clusters''. We introduce the notation
$$
\De_E(x,\,y):=\min\left\{\ell:\,\exists\,\{x_0=x,\,x_1,\,\ldots,\,x_\ell=y\}\subseteq  E,\,x_i\sim x_{i+1}\;\forall\,0\le i\le \ell-1,\,x_i\neq x_j\,\forall \,i\neq j\right\}
$$
for the graph distance on $E\subset\Z^d,\,x,\,y\in E$.
In the rest of the paper, $c=c(d)$ denotes a constant depending from the dimension which may vary from line to line. 
\begin{figure}[!ht]
\centering
\includegraphics[scale=0.8]{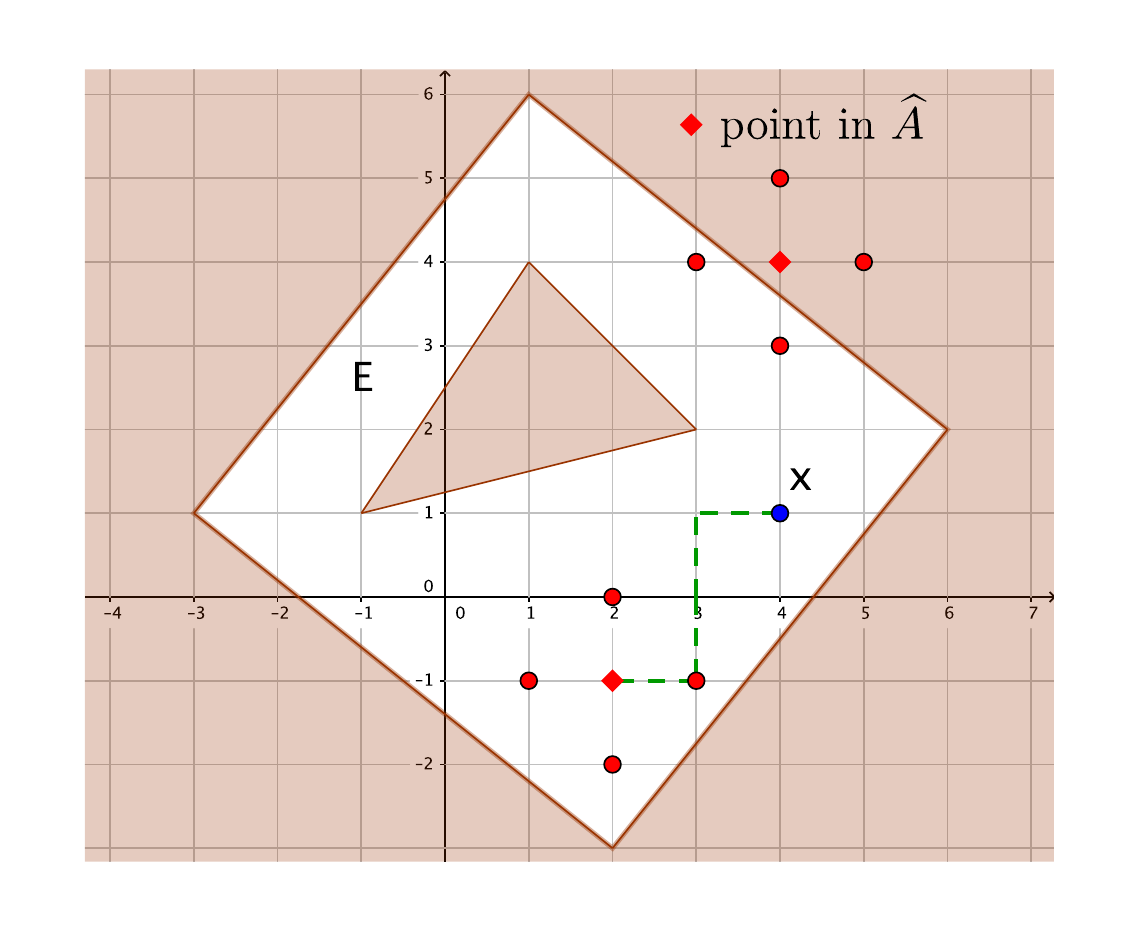}
\caption{$E$ in white. $A$ in red. The length of the green path is $\De_E(x,\,\widehat A\cap E).$}
\end{figure}
\begin{lemma}\label{lem:equiv}
Let $E\subset \Z^d$ be connected in the $\ell^1$-topology. Assume there exists $M<+\infty$ such that $\sup_{x\in E}\De_E\left(x,\,\widehat A\cap E\right)\leq \nicefrac{M}{2}.$\footnote{For $\tau\in E$, $\Lambda\subset\Z^d$, $\De_E(\tau,\Lambda):=\inf_{\lambda\in\Lambda}\De_E(\tau,\,\lambda)$.}.  Let $u:\Z^d\to\R$ be a function in $H^2(E)$ such that $u(x)=0$ for all $x\in A.$ Then there exists a constant $c=c(d)$ such that 
\[\|u\|_{H^2(E)}\leq c M^{2d-2}\|\nabla_2 u\|_{L^2(E)}.\]
\end{lemma}
\begin{proof}
Consider a partition $\mathcal E$ of $E$ made up by sets of diameter at most $M$ such that each set $B\in \mathcal E$ has no-empty intersection with $\widehat A$. In other words, every set in this partition contains at least one point $x_0\in \widehat{A}.$ Fix $B\in \mathcal E$, and fix $y\in B.$ Then we can find a path $z_0,\,...,\,z_K$ inside $B,$ such that $z_0=x_0,\, z_K=y, \,z_n\neq z_m$ for $n\neq m,$ $\|z_{n+1}-z_n\|=1$ for all $n,$ and $K\leq M.$ Then, since $u(x_0)=0,$ 
\[|u(y)|^2=\left|\sum_{n=1}^K u(z_n)-u(z_{n-1})\right|^2\leq K \sum_{n=1}^K\left|u(z_n)-u(z_{n-1})\right|^2\leq K \sum_{z\in B}\left|\nabla u(z)\right|^2.\]
We can do this for every $y\in B.$ Thus
\[\sum_{y\in B}|u(y)|^2\leq M^{d+1}\sum_{z\in B}|\nabla u(z)|^2.\]
Hence, summing over all $B\in\mathcal E$, we obtain
\[\|u\|^2_{L^2(E)}=\sum_{B\in \mathcal E}\sum_{y\in B}|u(y)|^2\leq M^{d+1}\sum_{B\in \mathcal E}\sum_{y\in B}|\nabla u(z)|^2=M^{d+1}\|\nabla u\|^2_{L^2(E)}.\]
Now we want to use the same type of argument on $|D_iu(y)|$ (resp. $|\nabla u(y)|$). Since $x_0\in \widehat A$, we have $\nabla u(x_0)=0$. So our argument gives 
\eqa{*}
|\nabla u(y)|^2&=\left|\nabla u(x_0)+\sum_{n=1}^K(\nabla u(z_n)-\nabla u(z_{n-1}))\right|^2\leq  K\sum_{z\in B}\sum_{i,j}|D_{ij} u(z)|^2
\eeqa{*}
which leads to 
\eqa{*}
\sum_{y\in B}|\nabla u(y)|^2&\le  M^{d+1}\sum_{z\in B}\sum_{i,\,j=1}^d    |D_{ij} u(z)|^2.
\eeqa{*}
Thus 
\[\|\nabla u\|^2_{L^2(E)}\leq M^{d+1}\|\nabla_2 u\|^2_{L^2(E)}.\]
Finally
\eqa{*}
\|u\|_{H^2(E)}^2=&\|u\|^2_{L^2(E)}+\|\nabla u\|^2_{L^2(E)}+\|\nabla_2 u\|^2_{L^2(E)}\le  \|\nabla u\|^2_{L^2(E)}\left(M^{d+1}+1\right)+\|\nabla_2 u\|^2_{L^2(E)}\\
\le&\!\!\!\left( M^{2(d+1)}+M^{d+1}+1\right)\|\nabla_2 u\|^2_{L^2(E)}\le c(d)M^{2(d+1)}\|\nabla_2 u\|^2_{L^2(E)}.
\eeqa{*}
This completes the proof.
\end{proof}

For fixed $y,$ let $B_k=B_{k,y}:=\left\{x\in \Z^d:\,\Vert x -y \Vert_1\leq k \right\}$ denote the ball with radius $k$ and center $y$ on the lattice, $k\geq 0.$ We denote by $G_A^y$ a solution to \eqref{eq:bv-prob} for fixed $y$. Recall in $d=4$ we are extending $G_A^y$ to $0$ outside $V_N\times V_N$.

\subsection{Deterministic pinning}\label{subsec:det_det}

The equivalence of norms of Lemma \ref{lem:equiv} can be applied as follows. 

\begin{lemma}\label{lem:mazya}
Let $d\geq 1$ and let $A\subset \Z^d$. Fix $k\geq 5.$ For any connected subset $D_k\subseteq B_{k}^c$ for which there exists $M=M(D_k,A)<+\infty$ such that $\sup_{x\in D_k} \De_{D_k}(x,\widehat{A}\cap D_k)\leq \nicefrac{M}{2},$ there exists $c=c(d)$ such that 
\[\|G_A^y\|_{H^2(D_k)}^2\leq c M^{2d+2}\|G_A^y\|_{H^2(B_k\setminus B_{k-5})}.\]
\end{lemma}

\begin{proof}
 Let $(\eta_k)_{k\geq 1}$ be a family of cutoff functions such that 
$$\left\{\begin{array}{lr}
\eta_k(x)=1& x\in B_{k-2}^{\c},\\
\eta_k(x)=0& x\in B_{k-3},\\
0\leq \eta_k(x)\leq 1 & x\in \Z^d.
\end{array}\right.$$
We also denote by $\eta_k G_A^y(x):=\eta_k(x)G_A^y(x)$ the pointwise product of the two functions. Since we have $\eta_kG_A^y=G_A^y$ on $B_{k-2}^c,$ we obtain from Lemma \ref{lem:equiv}
\begin{eqnarray}
\|G_A^y\|^2_{H^2(D_k)}&= &\|\eta_k G_A^y\|^2_{H^2(D_k)} \leq c(d) M^{2d+2} \|\nabla_2(\eta_k  G_A^y)\|^2_{L^2(D_k)} \nonumber \\
&\leq & c(d) M^{2d+2} \|\nabla_2(\eta_k  G_A^y)\|^2_{L^2(\Z^d)}.\label{eq:first}
\end{eqnarray}
On the other hand we have, using Lemma \ref{lem:laplace-norm}, the properties of $\eta_k$, and the fact that $G_A^y(x)$ is biharmonic,
\begin{eqnarray}
\|\nabla_2(\eta_k  G_A^y)\|^2_{L^2(\Z^d)}&=&\sum_{x\in \Z^d} \sum_{i,j=1}^d |D_iD_j (\eta_k G_A^y(x))|^2=4d^2\sum_{x\in \Z^d} \left(\eta_k G_A^y(x)\right)\left(\Delta^2\eta_k G_A^y(x)\right)\nonumber\\
&=& 4d^2\sum_{x\in B_{k-1}\setminus B_{k-3}} \left(\eta_k G_A^y(x)\right)\left(\Delta^2\eta_k G_A^y(x)\right)\nonumber \\
&\leq & C_2(d)\Vert \eta_k G_A^y\Vert^2_{H^2\left(B_{k}\setminus B_{k-5}\right)}\leq C_2(d) \Vert  G_A^y\Vert^2_{H^2\left(B_{k+1}\setminus B_{k-5}\right)}\label{eq:parts}.
\end{eqnarray}
Putting \eqref{eq:first} and \eqref{eq:parts} together gives the desired result. Observe that the above Lemma holds for any dimension $d\geq 1,$ in particular for $d=4,$ if we set 
\eq{}\label{eq:G_AN}
G_{A,N}^y:=G_{A\cap V_N}(\cdot,\, y),
\eeq{}
which we extend to $\Z^d$ by setting it $0$ outside $V_N.$ The statement is obviously  interesting only if $D_k\cap V_N\neq \emptyset,$ but it is trivially true otherwise.
\end{proof}

With this preparation, we can prove the following deterministic version of our main result, whose proof illustrates the ideas behind our approach.

\begin{theorem}[Deterministic pinned set]\label{thm:deterministic}
Let $d\geq 5,$ and let $A\subset \Z^d$ be such that there exists $M<+\infty$ such that $\sup_{x\in A^{\c}}\De(x,\,\widehat A)\leq \nicefrac{M}{2}.$ Then there exist $s=s(d,M)\in(0,\,+\infty)$ and $c=c(d, M)\in(0,\,+\infty)$ such that for all $x,\,y\in A^{\c}$
\[ |G_A(x,y)|\leq c \e^{-s \|x-y\|}\]
Moreover $\left\|G_A^y\right\|^2_{H^2(\Z^d)}\le \gamma$.
\end{theorem}

\begin{proof}
From Lemma \ref{lem:mazya} we obtain, choosing $D_k=B_k^c,$ and observing $\|G_A^y\|_{H^2(A\cup B)}^2=\|G_A^y\|_{H^2(A)}^2+\|G_A^y\|_{H^2(B)}^2$ for any disjoint sets $A,B\subset \Z^d,$

\[\Vert G_A^y\Vert^2_{H^2\left(B_k^c\right)}\leq cM^{2d+2}\left(\Vert G_A^y\Vert^2_{H^2\left(B_{k-5}^c\right)}-\Vert G_A^y\Vert^2_{H^2\left(B_k^c\right)}\right)\] 
which yields 
\[ \Vert G_A^y\Vert^2_{H^2\left(B_{k}^{\c}\right)}\leq \frac{c M^{2d+2}}{1+c M^{2d+2}}\Vert G_A^y\Vert^2_{H^2\left(B_{k-5}^{\c}\right)}.\]
Since $\Vert G_A^y\Vert^2_{H^2(B_{i}^{\c})}\geq \Vert G_A^y\Vert^2_{H^2(B_{i+1}^{\c})}$ for all $i\geq 0$, iteration yields, setting $C=cM^{2d+2},$
\begin{equation}\label{eq:bound_norm}\begin{split} \Vert G_A^y\Vert^2_{H^2\left(B_{k}^{\c}\right)}\leq & \left(\frac{C}{C+1}\right)^{\lfloor \nicefrac{k}{5}\rfloor} \Vert G_A^y\Vert^2_{H^2(B_{0}^{\c})}\leq \left(\Big(\frac{C}{C+1}\Big)^{1/5} \right)^{k-1}\Vert G_A^y\Vert^2_{H^2(B_{0}^{\c})}\\
\leq& \e^{-s(k-1)}\Vert G_A^y\Vert^2_{H^2\left(B_0^{\c}\right)}\end{split}
\end{equation}
for 
\[s=\frac{1}{5}\log\frac{1+C}{C}>0.\]
Note that for some $c(d)<+\infty$ we have
\[\Vert G_A^y\Vert^2_{H^2(B_0^{\c})}\leq \Vert G_A^y\Vert^2_{H^2(\Z^d)}=\sum_{z\in\Z^d}\left(\sum_{i,j=1}^d D_i D_j G_A^y(z)\right)^2\leq c(d) \sum_{z\in\Z^d}\sum_{i,j=1}^d \left(D_i D_j G_A^y(z)\right)^2.\]
Lemma \ref{lem:laplace-norm} and the fact that $G_A^y(z)=0$ for $z\in A$ give
\begin{equation}\label{eq:norm_value}\begin{split}
\sum_{z\in\Z^d}\sum_{i,j=1}^d \big(D_i D_j G_A^y(z)\big)^2&=\sum_{z\in\Z^d}\Delta^2 G_A^y(z) G_A^y(z)=\sum_{z\in A^{\c}}\Delta^2 G_A^y(z) G_A^y(z)\\
&=\sum_{z\in A^{\c}}\delta_{y}(z)G_A^y(z)=G_A^y(y)\leq G(y,y)=\gamma\stackrel{\eqref{eq:boundC}}{<}+\infty.
\end{split}
\end{equation}
Hence $\Vert G_A^y\Vert^2_{H^2(B_{0}^{\c})}<+\infty$ for all $A,$ and we get for $\Vert x-y\Vert>k$ from \eqref{eq:bound_norm}
\eq{}\label{eq:goal}
\left|G_A^y(x)\right|^2\leq  \Vert G_A^y\Vert^2_{H^2(B_{k}^{\c})}\leq C\e^{-sk}
\eeq{}
which is what we wanted to prove. \end{proof}
We now pass on proving the $4$-dimensional case as follows.
\begin{theorem}[Deterministic pinned set]\label{thm:deterministic_d4}
Let $d=4,$ and let $A\subset \Z^d$ be such that there exists $M<+\infty$ such that $\sup_{x\in A^{\c}}\De(x,\,\widehat A)\leq \nicefrac{M}{2}.$ Then there exist $s=s(d,M)\in(0,\,+\infty)$ and $c=c(d, M)\in(0,\,+\infty)$ such that for all $x,\,y\in A^{\c}$
\[ |G_{A\cap V_N}(x,y)|\leq c \log N \e^{-s \|x-y\|}.\]
Moreover $\left\|G_A^y\right\|^2_{H^2(\Z^d)}\le \gamma \log N.$
\end{theorem}
\proof
This works exactly as for Theorem \ref{thm:deterministic}, using $G_{A,N}^y$ defined in \eqref{eq:G_AN} instead of $G_A^y,$ except that \eqref{eq:norm_value} has to be replaced by 
\begin{equation}\label{eq:norm_value_d4}\begin{split}
\sum_{z\in\Z^d}\sum_{i,j=1}^d \big(D_i D_j G_A^y(z)\big)^2&=\sum_{z\in\Z^d}\Delta^2 G_{A,N}^y(z) G_{A,N}^y(z)=\sum_{z\in A^{\c}\cap V_N}\Delta^2 G_{A,N}^y(z) G_{A,N}^y(z)\\
&=\sum_{z\in A^{\c}\cap V_N}\delta_{y}(z)G_{A,N}^y(z)=G_{A,N}^y(y)\leq G_N(y,y)\stackrel{\eqref{eq:boundC}}{\leq}\gamma\log N.
\end{split}
\end{equation}
In the last inequality we have used Lemma~\ref{lem:G_A_d4}. This leads to the statement.
\endproof


\subsection{From deterministic to random pinning}\label{subsec:det_to_ran}

In this subsection we explain how to ``transfer'' the decay of covariances from the deterministic case to the random situation. Now the point is that in the random situation there is no fixed $M$ that we can take as in the previous proofs. Fix $k>5, A\subset\Z^d.$ The idea is now to choose sets $D_\ell^{(k)}, \,0\leq \ell \leq \lfloor \nicefrac{k}{5}\rfloor,$ in the right way such that there are suitable $M_\ell^{(k)}=M\left(D_\ell^{(k)}\right)$ for which we can adjust the iteration procedure. We make the following choices: 
\eq{}\label{eq:D}
D_\ell^{(k)}:=B_{k+1}\setminus B_{k-5\ell}, \quad 0\leq \ell\leq \lfloor \nicefrac{k}{5}\rfloor ,
\eeq{}
and
\[M_\ell^{(k)}=M_\ell^{(k)}(A):=\max_{x\in D_\ell^{(k)}}\De_{D_\ell^{(k)}}\left(x, \widehat{A}\cap D_\ell^{(k)}\right),\]
where the distance is taken on the lattice. If $D_\ell^{(k)}\cap \widehat{A}=\emptyset,$ let $M_\ell^{(k)}:=+\infty.$ 
The following Lemma, albeit deterministic, shows that if we wish to obtain a strong decay of correlations, one needs to control appropriately the maximal distance between a point and the clusters of pinned points. 
\begin{lemma}\label{lem:bound_random}
Let $d\geq 5.$ For $m_k=k^{\xi}$, $0< \xi<\nicefrac{1}{2(d+1)}$, define $a_k=a_k(A)>0$ as $a_k:=|\{\ell\in\{0\,,...,\,\lfloor \nicefrac{k}{5}\rfloor\}: M_\ell^{(k)}\leq m_k\}|$. Then there exist $c>0$ dependent only on $d$ such that for $\|x-y\|=k,$
\[|G_A^y(x)|\leq \gamma \e^{-c m_k^{-2(d+1)} a_k}\]
and $\gamma$ is as in \eqref{eq:boundC}.
\end{lemma}
\begin{proof} Observe that we have 
\[D_\ell^{(k)}\subseteq B_{k-5\ell}^c, \quad \mbox{and}\quad D_\ell^{(k)}\cup \left( B_{k-5\ell}\setminus B_{k-5(\ell+1)}\right)= D_{\ell+1}^{(k)}\]
where the last union is disjoint.
If $M_\ell^{(k)}<+\infty$ we thus get from Lemma \ref{lem:mazya} that
\[\|G_A^y\|^2_{H^2(D_\ell^{(k)})}\leq c (M_\ell^{(k)}) ^{2d+2}\|G_A^y\|^2_{H^2(B_{k-5\ell}\setminus B_{k-5\ell})}=c \left(M_\ell^{(k)}\right)^{2d+2}\left(\|G_A^y\|^2_{H^2\left(D_{\ell+1}^{(k)}\right)}-\|G_A^y\|^2_{H^2\left(D_\ell^{(k)}\right)}\right),\]
which leads to
\[\|G_A^y\|^2_{H^2\left(D_\ell^{(k)}\right)}\leq c\frac{\left(M_\ell^{(k)}\right)^{2d+2}}{1+c\left(M_\ell^{(k)}\right)^{2d+2}}\|G_A^y\|^2_{H^2(D^{(k)}_{\ell+1})}.\]
If $M_\ell^{(k)}=+\infty$ we have, since $D_\ell^{(k)}\subset D_{\ell+1}^{(k)},$
\[\|G_A^y\|^2_{H^2\left(D_\ell^{(k)}\right)}\leq \|G_A^y\|^2_{H^2\left(D^{(k)}_{\ell+1}\right)}.\]
Hence 
$$\|G_A^y\|^2_{H^2\left(D^{(k)}_\ell\right)}\leq \left(\one{M_\ell^{(k)}<+\infty}c\frac{(M_\ell^{(k)})^{2d+2}}{1+c M_\ell^{2d+2}}+ \one{M_\ell=+\infty}\right)\|G_A^y\|^2_{H^2\left(D^{(k)}_{\ell+1}\right)}$$ 
for all $0\le \ell\le \lfloor \nicefrac{k}{5}\rfloor$. Iteratively we find
\begin{align}
\|G_A^y\|^2_{H^2(B_{k+1}\setminus B_k)} &=\|G_A^y\|^2_{H^2\left(D^{(k)}_0\right)}\nonumber\\
&\leq \prod_{\ell=0}^{\lfloor \nicefrac{k}{5}\rfloor-1} \left(\one{(M_\ell^{(k)})<+\infty}c\frac{ (M_\ell^{(k)})^{2d+2}}{1+c (M_\ell^{(k)})^{2d+2}}+ \one{(M_\ell^{(k)})=+\infty}\right)\|G_A^y\|^2_{H^2\left(D_{\ell+1}^{(k)}\right)}\nonumber\\
&\leq \prod_{\ell=0}^{\lfloor \nicefrac{k}{5}\rfloor-1} \left(\one{(M_\ell^{(k)})<+\infty}c\frac{ (M_\ell^{(k)})^{2d+2}}{1+c (M_\ell^{(k)})^{2d+2}}+ \one{(M_\ell^{(k)})=+\infty}\right)\|G_A^y\|^2_{H^2(\Z^d)}.\label{eq:iteration}
\end{align}
With our definition of $a_k$, we can then rewrite \eqref{eq:iteration} as
\begin{equation*}
\|G_A^y\|^2_{H^2(B_{k+1}\setminus B_k)}\leq \left(c\frac{m_k^{2(d+1)}}{1+cm_k^{2(d+1)}}\right)^{a_k}\|G_A^y\|^2_{H^2(\Z^d)}.\label{eq:power_bound}
\end{equation*} 
Using the fact that $\log\nicefrac{1+x}{x}\geq \nicefrac{1}{x}$, $x>0,$ we obtain for $x\in \Z^d\setminus \{y\}$ and for $k$ such that $x\in B_{k+1}\setminus B_k,$
\begin{equation*}
\left|G_A^y(x)\right|\leq \|G_A^y\|^2_{H^2(B_{k+1}\setminus B_k)}\leq\left(\frac{cm_k^{2(d+1)}}{1+cm_k^{2(d+1)}}\right)^{a_k}\|G_A^y\|^2_{H^2(\Z^d)}\leq \gamma\e^{-c(m_k)^{-2(d+1)}a_k}
\end{equation*}
where we have concluded by means of Theorem~\ref{thm:deterministic}.
\end{proof}
The $4$-dimensional case is also at hand as follows:
\begin{lemma}\label{lem:bound_random_d4}
Let $d=4.$ For $m_k=k^{\xi}$, $0\le \xi<d-1$, set $a_k=a_k(A)>0$ such that $a_k:=|\{\ell\in\{0\,,...,\,\lfloor \nicefrac{k}{5}\rfloor\}: M_\ell^{(k)}\leq m_k\}|$. Then there exist $c>0$ dependent only on $d$ such that for $\|x-y\|=k,$
\[|G_A^y(x)|\leq \gamma_d \log N \e^{-c m_k^{-2(d+1)} a_k}.\]
\end{lemma}
\begin{proof}
The proof is the same of Lemma~\ref{lem:bound_random}, where in the very last step one uses Theorem~\ref{thm:deterministic_d4}.
\end{proof}

Thus in order to prove our main result, we will try to make $m_k^{-2(d+1)}a_k$ as large as possible. We first have the following auxiliary Lemma:



\begin{lemma}\label{lem:maxBer} Let $\nu$ be a Bernoulli site percolation measure on $\Z^d$ with $\nu(x\text{ is open})=\rho\in (0,\,1)$, $x\in \Z^d$. Let $\mathcal A$ be the set of open sites. Furthermore let $(m_k)_{k\in\N}$ and $a_k=a_k(\mathcal A)$ be defined as in Lemma \ref{lem:bound_random}. Then there exists $C=C(d)\in (0,\,+\infty)$ independent of $\mathcal A$ and $k$ such that
$$
\nu\left(a_k\leq \lfloor \nicefrac{k}{10} \rfloor\right)\leq  C k^{d+1}(1-\rho^{2d+1})^{\lfloor m_k/4\rfloor}.
$$
\end{lemma}
\begin{proof}
Recall $\widehat{\mathcal A}:=\left\{x\in \mathcal A:\,y\in \mathcal A\;\text{for all }y\sim x\right\}$. We have $\nu(x\in \widehat{\mathcal A})=\rho^{2d+1}$. We also observe that if $\|x-y\|>2$, the events $\{x\in \widehat{\mathcal A}\}$ and $\{y\in \widehat{\mathcal A}\}$ are independent. For any $t\in \N$ with $t\le \left|D_\ell^{(k)}\right|$, we have

\begin{align*}
\nu&\left(\De_{D_\ell^{(k)}}(x,\,\widehat{\mathcal A})\ge t\right)\\
& \le\nu\left(\exists\,\{x_0=x,\,\ldots,\,x_t\},\,x_i\in D_\ell^{(k)}\setminus \widehat{\mathcal A},\,\,x_i\sim x_{i+1}\,\forall\;0\le i\le t-1,\,x_i\neq x_j\,\forall\;i\neq j\right)\\
&\le\nu\left(x_0\notin \widehat{\mathcal A},\,x_3\not\in\widehat{\mathcal A},\,\ldots,\,x_{\lfloor t/4\rfloor}\notin \widehat{\mathcal A}\right)=\left(1-\rho^{2d+1}\right)^{\lfloor t/4\rfloor}
\end{align*}
by independence. By means of the FKG inequality \cite[Theorem 2.16]{GrimmettRC}, 

\begin{align}
\nu&\left(\max_{x\in B_\ell^{(k)}}\De_{D_\ell^{(k)}}(x,\,\widehat{\mathcal A})\ge t\right)= 1-\nu\left(\De_{D_\ell^{(k)}}(x_0,\,\widehat{\mathcal A})< t,\,\exists\,x_0\in D_\ell^{(k)}\right)\nonumber\\
&\le 1-\left(1-\left(1-\rho^{2d+1}\right)^{\lfloor t/4\rfloor}\right)^{\left|D_\ell^{(k)}\right|}\le \left|D_\ell^{(k)}\right| \left(1-\rho^{2d+1}\right)^{\lfloor t/4\rfloor}\nonumber\\
&\le \left|D_{k+1}\right|\left(1-\rho^{2d+1}\right)^{\lfloor t/4\rfloor}=\left(\sqrt{2}(k+1)\right)^d \left(1-\rho^{2d+1}\right)^{\lfloor t/4\rfloor}.\label{eq:star}
\end{align}

By the condition imposed on $m_k$ we can choose $k$ large such that $m_k\leq \left|D_\ell^{(k)}\right|$ for all $0\leq \ell\leq\lfloor \nicefrac{k}{10}\rfloor.$ Then 

\begin{align*}
\nu \left(a_k\leq \lfloor \nicefrac{k}{10} \rfloor\right)
&\leq \nu \left(\max_{0\leq \ell \leq \lfloor \nicefrac{k}{10}\rfloor}\max_{x\in D_\ell^{(k)}}\De_{D_\ell^{(k)}}(x,\,\widehat{\mathcal A})\ge m_k\right)\\
&\leq \left\lfloor \frac{k}{10}\right\rfloor(\sqrt{2}(k+1))^d \left(1-\rho^{2d+1}\right)^{\lfloor m_k/4\rfloor}.
\end{align*}
\end{proof}

\begin{proof}[Proof of Theorem \ref{thm:random}]
Take $x,\,y\in \Z^d$ and assume $\|x-y\|>k\in \N$. Using the expansion \eqref{eq:mixture}, Equation \eqref{eq:boundC} and Lemma \ref{lem:bound_random} we get
\begin{align*}
&\left|E_N^\eps(\vr_x\vr_y)\right|\le E_{\zeta_N^\eps}\left(\left|G_{\mathcal A\cup V_N^c}(x,y)\right|\one{a_k(\mathcal A)<\lfloor \nicefrac{k}{10}\rfloor}\right)+E_{\zeta_N^\eps}\left(\left|G_{ \mathcal A\cup V_N^c}(x,y)\right|\one{a_k(\mathcal A)\ge \lfloor \nicefrac{k}{10}\rfloor}\right)\\
&\le \gamma\,{\zeta_N^\eps}\left(a_k(\mathcal A)< \left\lfloor\frac{k}{10}\right\rfloor\right) +\gamma\sum_{A\subseteq V_N}\zeta_N^\eps\left(\mathcal A=A,\,a_k(A)\ge \left\lfloor\frac{k}{10}\right\rfloor\right)\e^{-c\lfloor\frac{k}{10}\rfloor m_k^{-2(d+1)}}\\
&\le \gamma\,{\zeta_N^\eps}\left(a_k(\mathcal A)< \left\lfloor\frac{k}{10}\right\rfloor\right)+\gamma\e^{-c\frac{k-10}{10} k^{-2\xi(d+1)}}.
\end{align*}
Since $\left\{a_k(\mathcal A)<\lfloor \nicefrac{k}{10}\rfloor\right\}$ is a decreasing event for the percolation realisation, we can use Proposition~\ref{prop:dom} to obtain
$$
\zeta_N^\eps\left(a_k(\mathcal A)<\lfloor \nicefrac{k}{10}\rfloor\right)\le \nu^{\rho_-(d)}\left(a_k(\mathcal A)<\lfloor \nicefrac{k}{10}\rfloor\right),
$$
where due to Lemma \ref{lem:maxBer} the right-hand side is bounded by $\e^{-k^{\xi'}},$ for any $\xi'<\xi.$ Thus we get the desired result for any $0<\alpha<\min\{\xi, 1-2\xi(d+1)\}.$
\end{proof} 
\begin{proof}[Proof of Theorem \ref{thm:random_d4}]
We can proceed as in the previous proof and obtain 
\begin{align*}
&\left|E_N^\eps(\vr_x\vr_y)\right|\le E_{\zeta_N^\eps}\left(\left|G_{\mathcal A\cup V_N^c}(x,y)\right|\one{a_k(\mathcal A)<\lfloor k/10\rfloor}\right)+E_{\zeta_N^\eps}\left(\left|G_{ \mathcal A\cup V_N^c}(x,y)\right|\one{a_k(\mathcal A)\ge \lfloor \nicefrac{k}{10}\rfloor}\right)\\
&\le \gamma\,{\zeta_N^\eps}\left(a_k(\mathcal A)< \left\lfloor\frac{k}{10}\right\rfloor\right) +\gamma\log N \sum_{A\subseteq V_N}\zeta_N^\eps\left(\mathcal A=A,\,a_k(A)\ge \left\lfloor\frac{k}{10}\right\rfloor\right)\e^{-c\left\lfloor\frac{k}{10}\right\rfloor m_k^{-2(d+1)}}\\
&\le \gamma\,{\zeta_N^\eps}\left(a_k(\mathcal A)< \left\lfloor\frac{k}{10}\right\rfloor\right)+\gamma\log N \e^{-c\frac{k-10}{10} k^{-2\xi(d+1)}}.
\end{align*}
We have to take care of the fact that $\rho_-$ converges to $0$ as $N\to +\infty.$ From Proposition \ref{prop:dom_d4} and Lemma \ref{lem:maxBer} we have
\[ \nu^{\rho_-(d)}\left(a_k(\mathcal A)<\lfloor \nicefrac{k}{10}\rfloor\right)\leq Ck^{d+1}\left(1-\frac{\eps c_-}{\sqrt{\log N}}\right)^{k^\xi/4}.\]
Inserting $k\geq N^\lambda,$ we thus get
\[\zeta_N^\eps\left(a_k(\mathcal A)<\lfloor \nicefrac{k}{10}\rfloor\right)\le \nu^{\rho_-(d)}\left(a_k(\mathcal A)<\lfloor \nicefrac{k}{10}\rfloor\right)\leq \e^{-\lambda \xi'}\]
for any $\xi'<\xi.$ Then we conclude by the same arguments as before.
\end{proof}

\section*{Acknowledgements}
We are grateful to Vladimir Maz'ya who gave us plenty of significant insights on how to show the exponential decay in a continuum version of the deterministic case. The second and third author also acknowledge the kind hospitality of the University of Zurich where a part of this research was carried out.

\bibliographystyle{abbrvnat}
\bibliography{biblio}

\begin{thebibliography}{15}
\providecommand{\natexlab}[1]{#1}
\providecommand{\url}[1]{\texttt{#1}}
\expandafter\ifx\csname urlstyle\endcsname\relax
  \providecommand{\doi}[1]{doi: #1}\else
  \providecommand{\doi}{doi: \begingroup \urlstyle{rm}\Url}\fi

\bibitem[Bolthausen and Brydges(2001)]{BoltBrydges}
E.~Bolthausen and D.~Brydges.
\newblock Localization and decay of correlations for a pinned lattice free
  field in dimension two.
\newblock \emph{IMS Lecture Notes Series}, \penalty0 (36):\penalty0 139--149,
  2001.

\bibitem[Bolthausen and Velenik(2001)]{BoltVel}
E.~Bolthausen and Y.~Velenik.
\newblock Critical behavior of the massless free field at the depinning
  transition.
\newblock \emph{Comm. Math. Phys.}, 223\penalty0 (1):\penalty0 161--203, 2001.

\bibitem[Caravenna and Deuschel(2008)]{CaravennaDeuschel_pin}
F.~Caravenna and J.-D. Deuschel.
\newblock Pinning and wetting transition for {$(1+1)$}-dimensional fields with
  {L}aplacian interaction.
\newblock \emph{Ann. Probab.}, 36\penalty0 (6):\penalty0 2388--2433, 2008.
\newblock ISSN 0091-1798.

\bibitem[Cipriani(2013)]{Cip13}
A.~Cipriani.
\newblock High points for the membrane model in the critical dimension.
\newblock \emph{Electron. J. Probab.}, 18:\penalty0 no. 86, 1--17, 2013.
\newblock ISSN 1083-6489.
\newblock \doi{10.1214/EJP.v18-2750}.
\newblock URL \url{http://ejp.ejpecp.org/article/view/2750}.

\bibitem[Coquille and Mi\l{}o\'s(2013)]{CoqMil}
L.~Coquille and P.~Mi\l{}o\'s.
\newblock {A note on the discrete Gaussian free field with disordered pinning
  on {$\mathbb Z^d,\,d\ge 2$}}.
\newblock \emph{Stochastic Processes and their Applications}, 123\penalty0
  (9):\penalty0 3542 -- 3559, 2013.
\newblock ISSN 0304-4149.
\newblock \doi{http://dx.doi.org/10.1016/j.spa.2013.04.022}.
\newblock URL
  \url{http://www.sciencedirect.com/science/article/pii/S0304414913001269}.

\bibitem[Deuschel and Velenik(2000)]{DeuschelVelenik}
J.~D. Deuschel and Y.~Velenik.
\newblock Non-gaussian surface pinned by a weak potential.
\newblock \emph{Probab. Theory Rel. Fields}, 116:\penalty0 359--377, 2000.

\bibitem[Grimmett(2006)]{GrimmettRC}
G.~R. Grimmett.
\newblock \emph{The random-cluster model}, volume 333.
\newblock Springer Science \& Business Media, 2006.

\bibitem[Ioffe and Velenik(2000)]{IofVel}
D.~Ioffe and Y.~Velenik.
\newblock A note on the decay of correlations under {$\delta$}-pinning.
\newblock \emph{Probability Theory and Related Fields}, 116\penalty0
  (3):\penalty0 379--389, 2000.
\newblock ISSN 0178-8051.
\newblock \doi{10.1007/s004400050254}.
\newblock URL \url{http://dx.doi.org/10.1007/s004400050254}.

\bibitem[Kurt(2008)]{Kurt_thesis}
N.~Kurt.
\newblock \emph{{Entropic repulsion for a Gaussian membrane model in the
  critical and supercritical dimension}}.
\newblock PhD thesis, University of Zurich, 2008.
\newblock URL \url{https://www.zora.uzh.ch/6319/3/DissKurt.pdf}.

\bibitem[Kurt(2009)]{Kurt_d4}
N.~Kurt.
\newblock {Maximum and entropic repulsion for a Gaussian membrane model in the
  critical dimension}.
\newblock \emph{The Annals of Probability}, 37\penalty0 (2):\penalty0 687--725,
  2009.

\bibitem[Sakagawa(2003)]{Sakagawa}
H.~Sakagawa.
\newblock {Entropic repulsion for a Gaussian lattice field with certain finite
  range interactions}.
\newblock \emph{J. Math. Phys.}, 44\penalty0 (7):\penalty0 2939--2951, 2003.

\bibitem[Sakagawa(2012)]{Sak2011}
H.~Sakagawa.
\newblock {On the Free Energy of a Gaussian Membrane Model with External
  Potentials}.
\newblock \emph{Journal of Statistical Physics}, 147\penalty0 (1):\penalty0
  18--34, 2012.
\newblock ISSN 0022-4715.
\newblock \doi{10.1007/s10955-012-0475-0}.
\newblock URL \url{http://dx.doi.org/10.1007/s10955-012-0475-0}.

\bibitem[Sznitman(2012)]{ASS}
A.-S. Sznitman.
\newblock \emph{{Topics in Occupation Times and Gaussian Free Fields}}.
\newblock Zurich Lectures in Advanced Mathematics. American Mathematical
  Society, 2012.
\newblock ISBN 9783037191095.
\newblock URL \url{http://books.google.ch/books?id=RnENO-nQ7TIC}.

\bibitem[Velenik(2006)]{velenikloc}
Y.~Velenik.
\newblock Localization and delocalization of random interfaces.
\newblock \emph{Probab. Surv}, 3:\penalty0 112--169, 2006.

\bibitem[Zhang(2006)]{zhang}
F.~Zhang.
\newblock \emph{The Schur Complement and Its Applications}.
\newblock Numerical Methods and Algorithms. Springer US, 2006.
\newblock ISBN 9780387242736.
\newblock URL \url{https://books.google.de/books?id=EMEyg8NcuskC}.

\end{thebibliography}
\end{document}